\documentclass[12pt]{article}
\usepackage{amsmath}
\usepackage{amsthm}
\usepackage{amsxtra}
\usepackage{amsfonts,amssymb}
\usepackage{graphicx}
\usepackage{marginnote}
\usepackage[top=1.5in, bottom=1.5in, left=1in, right=1in]{geometry}

\newtheorem{Th}{Theorem}[section]
\newtheorem{Prop}[Th]{Proposition}
\newtheorem{Lm}[Th]{Lemma}

\newtheorem*{mainthm}{Main Theorem}{\bf}{\it}

\def\be{\begin{eqnarray}} \def\ee{\end{eqnarray}} \def\bes{\begin{eqnarray*}}
\def\ees{\end{eqnarray*}} \def\bit{\begin{itemize}} \def\eit{\end{itemize}}
\def\ba{\begin{array}{ll}} \def\ea{\end{array}}

\def\mn{\mathbb{N}} 

\def\mz{\mathbb{Z}}

\def\mc{\mathbb{C}}  

\def\mr{\mathbb{R}}

\newtheorem{Def}[Th]{Definition}

   \def\mn{\mathbb{N}}
 \def\ge{\geqslant}  

\def\d{\operatorname{dist}}  
\def\diam{\operatorname{diam}}

\numberwithin{equation}{section}

\begin{document}
\title{Almost every real quadratic polynomial has a poly-time computable Julia set.}
\author{Artem Dudko and Michael Yampolsky}

\abovedisplayskip 8pt \abovedisplayskip 8pt \belowdisplayskip 9pt

\maketitle

\begin{abstract}
	We prove that Collet-Eckmann rational maps have poly-time computable Julia sets. As a consequence, almost all real quadratic Julia sets are poly-time.
\end{abstract}


\section{Introduction.}

A chaotic dynamical system can have a simple mathematical description, and thus be easy to implement on a computer. And yet, numerical simulation of its orbits is often impractical, since
small computational errors are magnified very rapidly. The modern paradigm of the numerical study of chaos can be summarized as follows:
while the simulation of an individual orbit for an extended
period of time does not make a practical sense, one should study the limit set of a typical orbit. Perhaps, the best known illustration of this approach is the study of Julia sets of rational maps which are repellers of the dynamics, that is, limit sets of typical backward orbits. Julia sets may be the most drawn objects in mathematics, and the study of the theoretical aspects of computing them is important both for practicing Complex Dynamicists, and as a ``simple'' test case of the paradigm.

This paper is motivated by the following general question, which we address in the context of quadratic Julia sets:

\medskip
\noindent
    {\bf Question 1.} {\sl Can the attractor/repeller of a typical dynamical system be efficiently simulated on a computer? }

    \medskip
    \noindent
To be more specific, let us first recall that
a compact set $K$ in the plane is computable if there exists an algorithm to draw it on
a computer screen with an arbitrarily high resolution. Any computer-generated picture is a finite collection of pixels.
If we fix a specific pixel size (commonly taken to be $2^{-n}$ for some $n$) then to accurately draw the set within
one pixel size, we should fill in the pixels which are close to the set (for instance, within distance $2^{-n}$ from it),
and leave blank the pixels which are far from it (for instance, at least $2^{-(n-1)}$-far). Thus, for the set $K$ to be
computable, there has to exist an algorithm which for every square of size $2^{-n}$ with dyadic rational vertices
correctly decides whether it should be filled in or not according to the above criteria.
We say that a computable set has a polynomial time complexity (is {\it poly-time})
if there is an algorithm which does this in a time bounded by a polynomial function of the precision parameter $n$, independent
of the choice of a pixel. We typically view poly-time computable sets as the ones which can be simulated efficiently in practice, indeed, in known applications,
this is generally the case.

When we talk of computability of the Julia sets of a rational map $R$,
the algorithm drawing it is supposed to have access to the values of the coefficients of the map (again with an arbitrarily
high precision). Using estabilished terminology such an algorithm can {\it query an oracle} for the coefficients of $R$; naturally, reading each additional binary
digit of a coefficient takes a single tick of the computer clock.

Computability of Julia sets has been explored in depth by M.~Braverman and the second author
(see monograph \cite{BY08} and references therein) and turns out to be a very non-trivial problem.
They have shown that even in the quadratic family $f_c(z)=z^2+c$ there exist values of $c$ such that the corresponding Julia sets $J_c$ are not computable.
Moreover, such a value of $c$ can be computed explicitly, and even, modulo a broadly accepted conjecture in one-dimensional dynamics, in polynomial time.

The phenomenon of non-computability is quite rare, and ``most'' quadratic Julia sets are computable. However, even a computable Julia set could have such a high computational complexity as to render any practical simulations impossible. Indeed,
in \cite{BBY06} it was shown that there exist
computable quadratic Julia sets with an arbitrarily high time complexity.

Restricted to the class of quadratic Julia sets, our first question transforms into:

\medskip
\noindent
{\bf Question 2.} {\sl Is it true that for almost every $c\in\mc$ the Julia set $J_c$ is poly-time?}

\medskip
\noindent
Poly-time computability has been previously established for several types of quadratic Julia sets.
Firstly, all hyperbolic Julia sets are
poly-time \cite{Bra04,Ret}. This theoretical result corresponds to a known efficient practical algorithm for such sets, developed by J.~Milnor \cite{Mil1} and known as {\it Milnor's  Distance Estimator}.

The requirement of hyperbolicity may be weakened significantly. Braverman \cite{Bra06} showed that parabolic quadratics also have poly-time computable Julia sets; and presented an efficient practical refinement of Distance Estimator for parabolics.
The first author showed in \cite{Dudko-14}
that maps with non-recurrent critical orbits have poly-time Julia sets.
Finally, in our previous joint work \cite{DudYam}, we have shown that Feigenbaum Julia set is poly-time. The last example is particularly interesting, since in contrast with the other ones mentioned above,  its proof does not use any weak hyperbolicity properties of the map itself, but rather a computational scheme based on self-similarity (Feigenbaum universality) properties of the Julia set.

Our main result gives a positive answer for Question 2 in the case of real parameters $c$:

\begin{mainthm}
For almost every real value of the parameter $c$, the Julia set $J_c$ is poly-time.
\end{mainthm}

Our principal technical result, which implies Main Theorem, establishes poly-time computability for Collet-Eckmann Julia sets (see the definitions below). Conjecturally, Collet-Eckmann parameters together with hyperbolic parameters form a set of full measure in $\mc$. This very strong conjecture would imply a positive answer to Question 2 as well (an even stronger form of this conjecture can be made about the parameter spaces of rational maps of degree $d$, $d\geq 2$, with similar consequences). However, at present, this is only known for $c\in\mr$.

Below we briefly recall the principal definitions of Computability Theory and Complex Dynamics, and then proceed with the proofs.

\subsection{Preliminaries on computability}
In this section we give a very brief review of computability and complexity of sets.
For details we refer the reader to the monograph \cite{BY08}.
 The notion of computability relies on the concept of a Turing Machine (TM) \cite{Tur},
 which is a commonly accepted way of formalizing
the definition of an algorithm.
The computational power of a Turing Machine is provably equivalent to that of a computer program running on a RAM
computer with an unlimited memory. We will use the terms ``TM'' and ``algorithm'' interchangeably.

\begin{Def}\label{comp fun def}
A function $f:\mn\rightarrow \mn$ is called computable, if there exists
a TM which takes $x$ as an input and outputs $f(x)$.
\end{Def}
Note that Definition \ref{comp fun def} can be naturally extended to
functions on arbitrary countable sets, using a convenient
identification with $\mathbb{N}$.


Let us denote $\mathcal D$ the set of {dyadic rationals}, that is, the set of rational numbers of the form $a/2^b$ where
$a\in\mz$ and $b\in\mn$.

To define computability of functions of real or complex variable we need to introduce the concept of an oracle:
\begin{Def}
A function $\phi:\mn\to\mathcal D+i\mathcal D$ is an oracle for $c\in\mc$ if for every $n\in\mn$ we have
$$|c-\phi(n)|<2^{-n} .$$
\end{Def}
A TM equipped with an oracle (or simply an {\it oracle TM}) may query the oracle by reading the value of $\phi(n)$ for an arbitrary $n$.
\begin{Def}
Let $S\subset \mc$. A function $f:S\to \mc$ is called computable if there exists an oracle TM $M^\phi$ with a single natural
input $n$ such that if $\phi$ is an oracle for $z\in S$ then $M^\phi(n)$ outputs $w\in \mathcal D +i\mathcal D$ such that
$$|w-f(z)|<2^{-n} .$$
\end{Def}
When calculating the running time of $M^\phi$, querying $\phi$ with precision $2^{-m}$ counts as
$m$ time units. In other words, it takes $m$ ticks of the clock to read the argument of $f$ with precision $m$ dyadic
digits (bits). This is, of course, in an agreement with the computing practice.

We say that a function $f$ is {\it poly-time computable} if in the above definition the algorithm $M^\phi$ can be made to run
in time bounded by a polynomial in $n$, independently of the choice of a point $z\in S$ or an oracle representing this
point.

 Let $d(\cdot,\cdot)$ stand for Euclidean distance between points or sets in $\mathbb{R}^2$.
 Recall the definition of the {\it Hausdorff distance} between two sets:
$$d_H(S,T)=\inf\{r>0:S\subset U_r(T),\;T\subset U_r(S)\},$$
where $U_r(T)$ stands for the $r$-neighborhood of $T$:
$$U_r(T)=\{z\in \mathbb{R}^2:d(z,T)\leqslant r\}.$$ We call
a set $T$ a {\it $2^{-n}$ approximation} of a bounded set $S$, if
$d_H(S,T)\leqslant 2^{-n}$. When we try to draw a $2^{-n}$ approximation $T$ of a set $S$
 using a computer program, it is convenient to
 let $T$ be a finite
 collection of disks of radius $2^{-n-2}$ centered at points of the form $(x,y)$
 with $x,y\in\mathcal D$.  We will call such a set {\it dyadic}.
A dyadic set $T$ can be described using a function
\begin{eqnarray}\label{comp fun}h_S(n,z)=\left\{\begin{array}{ll}1,&\text{if}\;\;d(z,S)\leqslant 2^{-n-2},
\\0,&\text{if}\;\;d(z,S)\geqslant 2\cdot 2^{-n-2},\\
0\;\text{or}\;1&\text{otherwise},
\end{array}\right.\end{eqnarray} where $n\in \mathbb{N}$ and $z=(i/2^{n+2},j/2^{n+2}),\;i,j\in \mathbb{Z}.$\\
 \\ Using this
function, we define computability and computational
complexity of a set in $\mathbb{R}^2$ in the following way.
\begin{Def}\label{DefComputeSet} A bounded set $S\subset
\mathbb{R}^2$ is called computable in time $t(n)$ if there
is a TM, which computes values of a function $h(n,\bullet)$
of the form (\ref{comp fun}) in time $t(n)$. We say that
$S$ is poly-time computable, if there exists a polynomial
$p(n)$, such that $S$ is computable in time $p(n)$.
\end{Def}
Computability and complexity of compact subsets of the Riemann sphere $\hat\mc$ are defined in a completely analogous fashion, substituting the Euclidean metric in the above with the standard spherical metric given by $dz/(1+|z|^2)$. Since the two metrics are equivalent on any compact subset of $\mc$, we have the following:
\begin{Prop}
\label{equiv-metric}
Let $S\Subset \mc$. Then computability of $S$ as a subset of $\mc$ in the Euclidean metric is equivalent to computability of $S$ as a subset of $\hat\mc$ in the spherical metric. Moreover, $S$ is poly-time in the former sense if and only if it is poly-time in the latter.
\end{Prop}
The proof is a trivial exercise and will be left to the reader.

\subsection{Collet-Eckmann maps and the statement of the principal result}
We recall that for a rational map $f:\mc\to\mc$ of degree $\deg f\geq 2$ its {\it Fatou set} $F_f$ is the domain of Lyapunov stability of the dynamics of $f$. That is, $F_f$ consists of points $z\in\mc$ for which there exists a neighborhood $U\ni z$ in which the sequence of iterates
$\{f_{|U}^n\}_{n\in\mn}$ is equicontinuous with respect to the spherical metric on $\hat\mc$. The complement  of $F_f$ is the Julia set $J_f$; it is an always non-empty compact subset of $\hat\mc$ which is fully invariant under $f$, that is $f^{-1}(J_f)=J_f$. The Julia set has non-empty interior if and only if it is equal to all of $\hat\mc$.

A rational map $f$ with $\deg f\geq 2$ is called {\it hyperbolic} if there exists a smooth Riemannian metric $\mu$ on an open neighborhood of $J_f$ such that $f$ is strictly expanding with respect to the corresponding Riemannian norm:
\begin{equation}\label{hypexp}||Df(z)||_\mu>1.\end{equation}
Such maps have a particularly tractable dynamics.
As was shown by Braverman \cite{Bra04} and Rettinger \cite{Ret}, hyperbolic Julia sets are poly-time computable. Note,  that a hyperbolic Julia set cannot contain any {\it critical points} of $f$, that is, points $c$ where $f'(c)=0$. In fact, an equivalent definition of hyperbolicity is that every critical point has an orbit which converges to an attracting cycle of $f$ (all such cycles are evidently in the Fatou set). Hyperbolicity of $f$ is thus an open condition in the parameter space of rational maps of degree $d\geq 2$.

It is the main open conjecture in the field of Complex Dynamics, that hyperbolic rational maps form a dense set in the parameter space of rational maps of degree $d\geq 2$. It is well-known, however, that the set of hyperbolic parameters does not have full measure in this space for any such $d$. A particularly useful class of rational maps which exhibits a weak version of the hyperbolic expansion  property (\ref{hypexp}) is given by the Collet-Eckmann condition described below:

\begin{Def} A non-hyperbolic rational map $f$ is called Collet-Eckmann if there exist constants $C,\gamma>0$
such that the following holds:
 for any critical point $c\in J_f$ of $f$ whose forward orbit
 does not contain any critical points one has:
 \begin{eqnarray}\label{CE} \left|Df^n(f(c))\right|\geqslant Ce^{\gamma n}\;\;\text{for any}\;\;n\in \mathbb{N}.
\end{eqnarray} \end{Def} In \cite{AvilaMoreira-05} Avila and Moreira showed:

\begin{Th}
  \label{avmor}
For almost every real parameter $c$ the map $f_c(z)=z^2+c$
 is either Collet-Eckmann or hyperbolic.

\end{Th}
In \cite{Aspenberg-13} Aspenberg proved that the set of Collet-Eckmann parameters
has positive Lebesgue measure in the space of coefficients of all rational
maps of fixed degree $d\geqslant 2$. Moreover, there is a conjecture
that {\it almost all parameters in this space correspond to either Collet-Eckmann or hyperbolic maps}.

The following property can be viewed as a form of weak hyperbolicity for a rational mapping:
\begin{Def}\label{DefESC} A rational map $f$ satisfies Exponential Shrinking of Components (ESC) condition if there exists $\lambda<1$ and $r>0$ such that for every $n\in\mathbb N$, any $x\in J_f$ and any connected component $W$ of $f^{-n}(U_r(x))$ one has $\diam(W)<\lambda^n$.
\end{Def}

In a fundamental paper on Collet-Eckmann dynamics, Przytycki, Rivera-Letelier, and Smirnov showed \cite{PrzytyckiLetelierSmirnov-03}:
\begin{Th}
  \label{ThESC}
  Collet-Eckmann condition implies Exponential Shrinking of Components condition.
\end{Th}
It is elementary to see that
ESC implies that $f$ does not have rotational domains or parabolic periodic points. Jointly with Binder and Braverman, the second author has shown (cf.\cite{BBY07,BY08}):

\begin{Th}\label{ThComp} Let $f$ be a rational map without rotation domains. Then its Julia set is computable in the spherical metric by an oracle Turing machine $\mathcal M^\phi$ with an oracle representing the coefficients of $f$. The algorithm uses non-uniform information on parabolic periodic points of $f$ (which is sufficient to lower-compute the parabolic basins).
\end{Th}
We also use the following result (see  \cite{Przytycki-93}):
\begin{Prop}\label{PropCritRec} For any rational map $f$ there exists $\mu>0$ such that for any critical point $c$ in $J_f$ and any $n\in\mathbb N$ one has $|f^n(c)-c|\geqslant \mu^n$.
\end{Prop}
Our principal result is the following:
\begin{Th}\label{ThMain}
For each $d\geqslant 2$ there exists an oracle Turing Machine $\mathcal M_d^\phi$ with an oracle for the coefficients of a rational map $f$ of degree $\deg f=d$ satisfying ESC, such that the following holds.
 Given the  non-uniform information:
\begin{itemize}
\item dyadic numbers $\lambda$ and $r$ for which the conditions of Definition \ref{DefESC} hold
 and such that $U_{2r}(J_f)\setminus J_f$ does not contain any critical points of $f$,
\item and a dyadic number $\mu$ which satisfies the statement of Proposition \ref{PropCritRec},
\end{itemize}
$M^\phi$ computes $J_f$ in polynomial time.
\end{Th}
Theorem~\ref{ThMain}, together with Theorems~\ref{ThESC} and~\ref{avmor}, imply the Main Theorem.

Note that if $f$ has no attracting cycles, then Fatou-Sullivan classification implies that $f$ has no Fatou components, and hence $J_f=\hat\mc$, so the proof of Theorem~\ref{ThMain} becomes a triviality in this case. We thus assume that $f$ has at least one attracting cycle in $\hat\mc$, which, in particular, covers the case when $f$ is a polynomial (characterized by $f^{-1}(\infty)=\infty$, so that $\infty$ is a super-attracting point). Moreover, if $\infty \in J_f$ we can construct a dyadic point $a\in\mathbb C$ belonging to some attracting basin of $f$ and consider the map $f_a=h_a^{-1}\circ f\circ h_a$, where $h_a(z)=\frac{1}{z}+a$. Then $\infty\notin J_{f_a}$. Therefore, without loss of generality, we can assume that $\infty\notin J_f$. In view of Proposition~\ref{equiv-metric}, we can thus prove the statement of Theorem~\ref{ThMain} with respect to computability in the Euclidean metric in $\mc$.

As a preparatory step  in the proof of Theorem~\ref{ThMain}, using Theorem~\ref{ThComp} we construct a positive dyadic number $\epsilon<r$ and a dyadic neighborhood $U$ of $J_f$ such that:
\begin{equation}\label{EqUcond}U_{2\epsilon}(J_f)\subset U\;\;\text{and}\;\;f(U_\epsilon(U))\subset U_{r}(J_f).\end{equation}

\section{Proof of Theorem~\ref{ThMain}.}

\subsection{Distortion bounds.}
We will use the classical
Koebe One-Quarter Theorem (see e.g. \cite{Con}):
\begin{Th}\label{Koebe quater} Suppose
$f:U_r(z)\rightarrow \mathbb{C}$ is a univalent function. Then the image
$f(U_r(z))$ contains the disk of radius
${\frac{1}{4}}r|f'(z)|$ centered at
$f(z)$. \end{Th}



\noindent
Recall that the postcritical set of a rational map is defined as the closure of the union of the orbits of its critical points.
First, let us prove the following technical statement:
\begin{Lm}\label{LmGenESC} Assume that a rational map $f$ such that $\infty\notin J_f$ satisfies ESC  and $\lambda,r$ are the corresponding constants (see Definition \ref{DefESC}). Then there is an algorithm which given an oracle for $f$ computes dyadic numbers $\alpha,\beta>0$ such that for any $x\in J_f$, $r>\delta>0$, $n\in\mathbb N$ and any connected component $W$ of $f^{-n}(U_\delta(x))$ one has:
$$\operatorname{diam}(W)\leqslant \alpha\delta^\beta\lambda^n.$$
\end{Lm}
\begin{proof} Let $R=\sup\{|Df(z)|:z\in U_r(J_f)\},N=[\log_R \frac{r}{\delta}]$, where $[a]$ stands for the maximal integer less or equal to $a$. Then $f^N(U_\delta(x))\subset U_r(f^N(x))$. The inequality of Lemma \ref{LmGenESC} now follows from applying Definition \ref{DefESC} to $U_r(f^N(x))$.
\end{proof}
We need a generalization of Koebe Distortion Theorem for maps with critical points. From Lemma 2.1 from \cite{PrzytyckiRohde-98} we deduce the following:
\begin{Prop}\label{PropKoebeGen} For each $D\in\mathbb N$ there exists a constant $C>0$ such that the following is true. Let $W\subset \mathbb C$ be a domain and $f:W\to U_1(0)$ be a holomorphic map of degree at most $D$. Then for any $0<t\leqslant\frac{1}{2}$ and any $y\in f^{-1}(0)$ for the component $W'\ni y$ of $f^{-1}(U_t(0))$ one has:
$$\operatorname{diam}(W')\leqslant \frac{Ct}{|f'(y)|}.$$
\end{Prop}

\begin{Prop}\label{PropKoebeCE} Let $U$, $\epsilon$ be as in (\ref{EqUcond}). There is an algorithm computing dyadic
constants $K_1,K_2,C>0$
 such that for any
$z\in U$ and any $k\in \mathbb{N}$ if $f^k(z)\in U_r(J_f)\setminus
U_\epsilon(J_f)$ then one has $$ \frac{K_1}{|Df^k(z)|}\leqslant
\operatorname{dist}(z,J_f)\leqslant \frac{K_2C^{\sqrt k}}{|Df^k(z)|}.$$ \end{Prop}
\begin{proof} 
\begin{figure}\centering\includegraphics[width=0.7\textwidth]{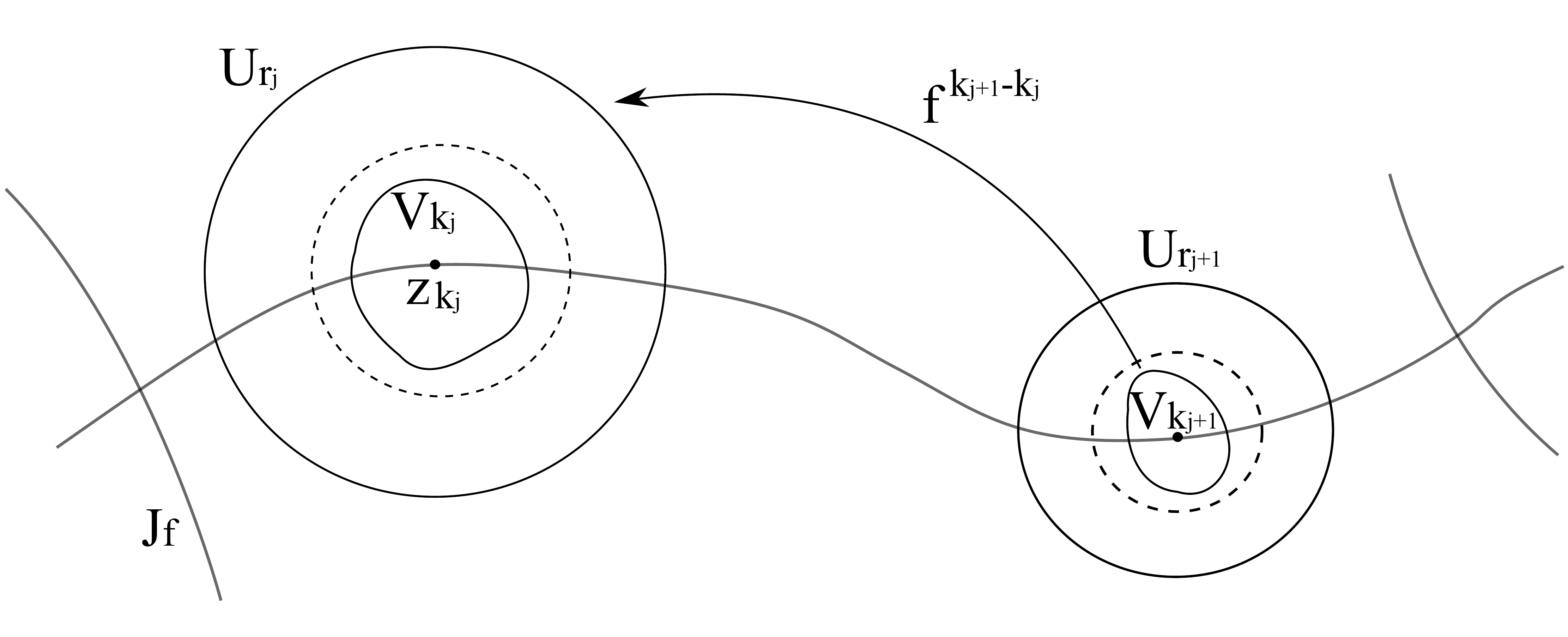}\caption{Sets $V_{k_j}$.} \end{figure}
Let $k,z$ be as in the conditions of Proposition \ref{PropKoebeCE}. Then the disk $U_\epsilon(f^k(z))$ does not intersect the postcritical set of $f$. Applying Koebe Quarter Theorem \ref{Koebe quater} to the inverse branch of $f^{-k}$ in this disk, we obtain $$\d(z,J_f)\geqslant \frac{\epsilon}{4|Df^k(z)|}.$$ 
Set $z_j=f^j(z)$ and $r_j=2^{-j}r$ for all $j\geqslant 0$. Let $W_0$ be the connected component of $f^{-k}(U_r(z_k))$ containing $z$. Set $W_l=f^l(W_0)$ for $0\leqslant l\leqslant k$. Fix the smallest $N_0\in\mathbb N$ such that
$$\lambda^{N_0}<\min\{2^{\beta-1},\tfrac{1}{4}\alpha^{-1}r^{1-\beta}\}.$$ Lemma \ref{LmGenESC} implies that for any $x\in J_f$ and any $j\in\mathbb Z_+$ the connected component of $f^{-N_0(j+1)}(U_{r_j}(f^{N_0(j+1)}(x)))$ containing $x$ has diameter less than
$$\alpha r_j^\beta\lambda^{N_0(j+1)}=\alpha r^\beta\lambda^{N_0}(2^{-\beta}\lambda^{N_0})^j<\tfrac{1}{4}r\cdot 2^{-j}=0.5r_{j+1}$$ and therefore is a subset of $U_{0.5r_{j+1}}(x)$.
 Introduce indexes $$k=k_0>k_1>k_2>\ldots>k_l=0$$ 
such that $$k_{j+1}=k_j-N_0(j+1)\text{ if }k_j\geqslant N_0(j+1)$$ and $k_{j+1}=0$ otherwise. Observe that $l=O(\sqrt k)$.

 Let $0\leqslant j<l$. For $k_{j+1}\leqslant i<k_j$ let $V_i$ be the connected component of $f^{i-k_j}(U_{r_j}(z_{k_{j}}))$ containing $z_{i}$. By the construction of the sequence $k_j$ we have $$V_{k_{j}}\subset U_{0.5r_j}(z_{k_j})\;\;\text{for all}\;\;j<l.$$ Moreover, Lemma \ref{LmGenESC} and definition of $r_j$ imply that there exists $C_1>0,\gamma<1$ depending only on $\alpha,\beta,\lambda$ and $r$ such that for all $k_{j+1}\leqslant i<k_j$ one has $\diam(V_i)<C_1\gamma^j$. Assume that for some indexes $k_{j+1}\leqslant i_1<i_2<k_j$ both $V_{i_1}$ and $V_{i_2}$ contain the same critical point $c$. Since $f^{i_2-i_1}(V_{i_1})=V_{i_2}$ using Proposition \ref{PropCritRec} we obtain that $$\mu^{i_2-i_1}\leqslant |f^{i_2-i_1}(c)-c|<C_1\gamma^j\;\;\text{and so}\;\;i_2-i_1>\log_\mu C_1+j\log_\mu \gamma.$$
It follows that the number of times $V_i$ contains a critical point for $k_{j+1}\leqslant i<k_j$ is bounded by some constant $M=M(C_1,\gamma,\mu,r,\deg f)$. Therefore, the degree of $f^{k_j-k_{j+1}}:V_{k_{j+1}}\to U_{r_j}(z_{k_j})$ is bounded by $N_1=(2\deg f-1)^M$. Since $W_{k_j}\subset U_{0.5r_j}(z_{k_j})$ for all $j<l$ using Proposition \ref{PropKoebeGen} we obtain that
$$\diam (W_{k_{j+1}})\leqslant \frac{C\cdot\diam (W_{k_j})}{|Df^{k_j-k_{j+1}}(z_{k_j})|}$$ for some constant $C$. Taking product of the latter inequality for all $0\leqslant j<l$ we obtain that
$$\diam(W_0)\leqslant \frac{C_2^{\sqrt k}r}{|Df^{k}(z)|},$$ where $C_2$ can be computed given $\lambda,r$ and an oracle for $f$. This finishes the proof.
\end{proof}

\subsection{The algorithm}

 Let $f$ be a rational map satisfying \text{ESC} and $\lambda, r$ be as in Definition \ref{DefESC}. Let $U, \epsilon$ be as in \eqref{EqUcond}. Assume that we would like to verify that a dyadic point $z$ is $2^{-n-1}$ close to $J_f$.
  If $z\notin  U$, we can approximate the distance from $z$ to $J_f$ by $\d(z,U)+r$ up to a constant factor.

Now assume that $z\in U$. Consider the
following subprogram:\\$i:=1$ \\{\bf while} $i\leqslant L(n)=-[(n+1)\log_\lambda 2+1]$
{\bf do}\\ $(1)$ Compute dyadic approximations $$p_i\approx
f^i(z)=f(f^{i-1}(z))\;\;\text{and}\;\;d_i\approx
\left|Df^i(z)\right|=\left|Df^{i-1}(z)\cdot
Df(f^{i-1}(z))\right|$$ \\with precision
$\min\{2^{-n-1},\epsilon\}$.\\ $(2)$ Check the inclusion
$p_i\in U$:\begin{itemize}\item[$\bullet$] if $p_i\in
U$, go to step $(5)$;\item[$\bullet$] if $p_i\notin U$,
proceed to step $(3)$;\end{itemize}
$(3)$ Check the inequality $d_i\geqslant K_2 C^{\sqrt{i}}2^{n+1}+1$. If true, output $0$ and exit the subprogram, otherwise\\
$(4)$ output $1$ and exit subprogram.\\
$(5)$ $i\rightarrow i+1$\\
{\bf end while}\\
$(6)$ Output $0$ end exit.\\
{\bf end}

\medskip

\noindent
The subprogram runs for at most $$L=-[(n+1)\log_\lambda 2+1]=O(n)$$
number of while-cycles each of which consist of a constant number of arithmetic operations with
precision $O(n)$ dyadic bits. Hence the running time of the subprogram can be bounded by $O(n^2\log
n\log\log n)$ using efficient multiplication. \begin{Prop}\label{PropAlgOutput} Let $h(n,z)$ be the
output of the subprogram. Then \begin{equation}h(n,z)=\left\{\begin{array}{ll}
1,&\text{if}\;\;d(z,J_f)>2^{-n-1},\\
0,&\text{if}\;\;d(z,J_f)<K2^{-n-1},\\\text{either}\;0\;\text{or}\;1,&\text{otherwise},\end{array}
\right.\end{equation}where $$K=K(n)=\displaystyle\frac{K_1}{K_2C^{\sqrt{L(n)}}+1},$$ \end{Prop}\begin{proof} Suppose first that the subprogram
runs the while-cycle $L$ times and exits at the step $(6)$. This means that $p_i\in U$ for
$i=1,\ldots,L$. In particular, $p_{L-1}\in U$. It follows that $f^{L}(z)\in U_r(J_f)$. By ESC condition we obtain: $$\d(z,J_f)\leqslant \lambda^{-L}\leqslant 2^{-n-1}.$$ Thus if $d(z,J_f)>
2^{-n-1}$, then the subprogram exits at a step other than $(6)$.

Now assume that for some $i\leqslant L$ the subprogram falls into the step $(3)$. Then $$p_{i-1}\in
U\;\; \text{and}\;\; p_i\notin U.$$ Conditions \eqref{EqUcond} imply that $f^i(z)\in U_r(J_f)\setminus U_\epsilon(J_f)$. Now, if $$d_i\ge
K_2C^{\sqrt{i}} 2^{n+1}+1\text{, then }|Df^i(z)|\geqslant K_2C^{\sqrt{i}}2^{n+1}.$$ By Proposition \ref{PropKoebeCE}, $$d(z,J_f)\leqslant
2^{-n-1}.$$ Otherwise, when the algorithm reaches step (4),
$$|Df^i(z)|\leqslant K_2C^{\sqrt{i}}2^{n+1}+2\leqslant (K_2C^{\sqrt{i}}+1)2^{n+1}.$$ In this case Proposition
\ref{PropKoebeCE}
 implies that
 $$d(z,J_f)\geqslant \frac{K_1}{K_2C^{\sqrt{L(n)}}+1}2^{-n-1}.$$
\end{proof}
 Now, to distinguish the case when $d(z,J_f)<2^{-n-1}$ from the case when
$d(z,J_f)>2^{-n}$ we can partition each pixel of size $2^{-n}\times 2^{-n}$ into pixels of size
$(2^{-n}/K)\times (2^{-n}/K)$ and  run the subprogram for the center of each subpixel. This would
increase the running time at most by a factor linear in $n$.

\end{document}